\documentclass[11pt,reqno]{amsart}
\usepackage{amssymb,latexsym}
\usepackage{graphicx}
\usepackage{url}		

\usepackage{hyperref}
\hypersetup{colorlinks,citecolor=black,filecolor=black,linkcolor=black,urlcolor=black}

\calclayout

\makeatletter
\newcommand{\monthyear}[1]{%
  \def\@monthyear{\uppercase{#1}}}
\makeatother

\newcommand{\Q}{{\mathbb Q}}

\newcommand{\N}{{\mathbb N}}
\newcommand{\Z}{{\mathbb Z}}

\newcommand{\paren}[1]{\left( #1 \right)}
\newcommand{\cro}[1]{\left[ #1 \right]}
\newcommand{\abs}[1]{\left| #1 \right|}
\newcommand{\aco}[1]{\left\{ #1 \right\}}

\newcommand{\lcm}{\mathrm{lcm}}

\newcommand{\im}{\mathrm{im}}
\renewcommand{\det}{\mathrm{det}}

\theoremstyle{plain}
\numberwithin{equation}{section}
\newtheorem{thm}{Theorem}[section]
\newtheorem{theorem}[thm]{Theorem}
\newtheorem*{theorem*}{Theorem}
\newtheorem{lemma}[thm]{Lemma}

\newtheorem{definition}[thm]{Definition}
\newtheorem{proposition}[thm]{Proposition}
\newtheorem{corollary}[thm]{Corollary}
\newtheorem{remark}[thm]{Remark}

\begin{document}
\monthyear{May 2020}
\setcounter{page}{1}

\title{Cycles of Sums of Integers}
\author{Bruno Dular}
\email{bruno.dular@gmail.com}

\begin{abstract}
We study the period of the linear map $T:\Z_m^n\rightarrow \Z_m^n:(a_0,\dots,a_{n-1})\mapsto(a_0+a_1,\dots,a_{n-1}+a_0)$ as a function of $m$ and $n$, where $\Z_m$ stands for the ring of integers modulo $m$. Since this map is a variant of the Ducci sequence, several known results are adapted in the context of $T$. The main theorem of this paper states that the period modulo $m$ can be deduced from the prime factorization of $m$ and the periods of its prime factors. We also characterize the tuples that belong to a cycle when $m$ is prime.
\end{abstract}

\maketitle

\tableofcontents

\section{Introduction}

The aim of this paper\footnote{Published in: The Fibonacci Quarterly, Volume 58 No 2 (2020) 126.} is to study a variant of the well-known \emph{Ducci game of differences}. In this game, one starts with a $n$-tuple of integers and iterates the \emph{Ducci map} $(a_0,a_1,\dots,a_{n-1})\mapsto (|a_0-a_1|,|a_1-a_2|,\dots,|a_{n-1}-a_0|)$ to generate a \emph{Ducci sequence}. This process suggests the name \emph{Cycles of differences of integers} \cite{Lu}, which inspired the name of the present paper.

If $n$ is a power of $2$, we know that every Ducci sequence eventually vanishes, i.e., reaches the zero $n$-tuple. Else, a Ducci sequence will either vanish or enter a periodic cycle. As several authors pointed out \cite{Lu,Eh}, studying the latter case comes down to considering n-tuples that consist only of $0$ and $1$'s. Hence the Ducci map can be considered to be a linear map over $\Z_2^n$. This map can be generalized as stated in Definition \ref{def:T}, performing sums modulo $m$ for some positive integer $m$.

This variant has been introduced by Wong in \cite{Wo} and has been extensively studied by F. Breuer in \cite{Br2}, who noticed a link between Ducci sequences and cyclotomic polynomials. The results we prove here are similar, but we use an elementary approach, which allows us to solve the inseparable case, i.e., $m=p^k$ with $p$ a prime and $n$ divisible by $p$. This is the object of section \ref{sec:prime_powers}. There we also see that our method does not generalize to certain special cases ($p=2$ or $p$ a Wieferich prime), for which F. Breuer's method works.

Here $\N$ denotes the set of positive integers (with the convention $0\notin \N$).

\begin{definition}\label{def:T}
Let $m,n\in\N$. Let $$T:\Z_m^n\rightarrow\Z_m^n:\textbf{a}=(a_0,\dots,a_{n-1})\mapsto T\textbf{a}=(a_0+a_1,a_1+a_2,\dots,a_{n-1}+a_0).$$
A \emph{T-sequence} of $\Z_m^n$ is a sequence of the form $(T^r\textbf{a})_{r\geq 0}$ where $\textbf{a}\in\Z_m^n$ and it is said to be generated by the tuple $\textbf{a}$.

The tuple $\textbf{e}=(1,0,\dots,0)\in\Z_m^n$ and the T-sequence it generates are respectively called the \emph{basic tuple} and the \emph{basic T-sequence} of $\Z_m^n$.
\end{definition}

For example, the T-sequence generated by the basic tuple of $\Z_{10}^4$ starts as shown below. Note that the tuple $(2,4,6,4)$ repeats, hence the T-sequence becomes periodic at that point, with a cycle length of $4$. With the notations introduced later, we write $P(10,4)=4$.

\begin{center}
    \begin{tabular}{c c c c c c c}
    1 0 0 0 & $\mapsto$ & 1 0 0 1 & $\mapsto$ & 1 0 1 2 & $\mapsto$ \\
    1 1 3 3 & $\mapsto$ & 2 4 6 4 & $\mapsto$ & 6 0 0 6 & $\mapsto$ \\
    6 0 6 2 & $\mapsto$ & 6 6 8 8 & $\mapsto$ & 2 4 6 4 & $\mapsto$ & $\dots$
    \end{tabular}
\end{center}

\begin{remark}\label{rem:subspace}
Let $\textbf{a}$ be a tuple of $\Z_m^n$ where $m=dm'$ for some integers $d$ and $m'$. We consider $\textbf{a}$ as an element of $\Z_{m'}^n$ by identifying it with the element $\textbf{a}\mod m'$ of $\Z_{m'}^n$.
\end{remark}

As $x+y\equiv x-y\pmod{2}$, note that T-sequences are Ducci sequences when $m=2$. Several known results can then be generalized.

To simplify notation, the components of a tuple $\textbf{a}\in\Z^n_m$ are indexed from $0$ to $n-1$. We sometimes write $[\textbf{a}]_i$ for $\textbf{a}_i$. Note that addition and subtraction of the indices will always be performed modulo $n$. Thus, $\textbf{a}_i$ should be understood as $\textbf{a}_{(i\mod n)}$.

Since $\Z_m^n$ is finite, a T-sequence must be eventually periodic. The goal of this paper is to study the maximal cycle length as a function of $m$ and $n$, which we denote by $P(m,n)$. We shall detail what we mean by the length of the period.

\begin{definition}\label{def:period}
Given $m$, $n$ and $\textbf{a}\in\Z_m^n$, a positive integer $L$ is the \emph{cycle length} of the T-sequence $(T^r\textbf{a})_{r\geq 0}$ if the following conditions hold:
\begin{enumerate}
    \item There exists a positive integer $N$ such that $T^{r+L}\textbf{a}=T^r\textbf{a}$ for all $r\geq N$.
    \item Every positive integer $L'$ satisfying $T^{r+L'}\textbf{a}=T^r\textbf{a}$ for large enough $r$ is a multiple of $L$.
\end{enumerate}
The smallest such $N$ is called the \emph{pre-period}. If $N$ is the pre-period, then the finite sequences $(\textbf{a},T\textbf{a},\dots,T^{N-1}\textbf{a})$ and $(T^N\textbf{a},\dots,T^{N+L-1}\textbf{a})$ are called \emph{pre-cycle} and \emph{cycle}, respectively.

In other words, the cycle length of the T-sequence $(T^r\textbf{a})_{r\geq 0}$ is the smallest positive integer $L$ such that there exists some $N\in\N$ satisfying $T^{r+L}\textbf{a}=T^r\textbf{a}$ for all $r\geq N$.

We define $\mathcal{C}_m^n$ as the subset of $\Z_m^n$ of all tuples that belong to a cycle. It directly follows from Remark \ref{rem:subspace} that $\mathcal{C}_m^n\subset\mathcal{C}_d^n$ whenever $d$ divides $m$.
\end{definition}

To simplify notation, cyclic permutations of a cycle are also called \emph{cycles} and thus we refer to \emph{a} cycle than \emph{the} cycle.

In Section \ref{sec:characterization_cycles}, we give a characterization of $\mathcal{C}_m^n$ for some $m$ and $n$.

\begin{definition}
Given $m$ and $n$ in $\N$, we define $$P(m,n)=\max\{P\in\N : P\textrm{ is the cycle length of some T-sequence in }\Z_m^n\}.$$
It is called the \emph{period} of T-sequences in $\Z_m^n$. This defines a function $P:\N\times\N\rightarrow\N$ called the \emph{period function}.
\end{definition}

As Ehrlich pointed out in \cite{Eh}, studying the cycle length of basic T-sequences suffices to determine the period of T-sequences. Actually, Ehrlich proved this result for Ducci sequences but the proof is essentially the same for T-sequences.

\begin{proposition}\label{prop:basic}
For all $m,n\in\N$, the cycle length of the basic T-sequence of $\Z_m^n$ equals $P(m,n)$. Cycle lengths of other T-sequences in $\Z_m^n$ divide $P(m,n)$.
\end{proposition}

In Section \ref{sec:basic_results} we give basic results that are useful to study more interesting properties of T-sequences. Among those we prove a generalization of the known fact that Ducci sequences of $2^n$-tuples eventually vanish.

In Sections \ref{sec:order_2_Wieferich} and \ref{sec:prime_powers} we give important theorems about the multiplicity of the period function, which are summed up in the following theorem. It is the main result of this paper.

\begin{theorem*}\label{thm:main}
Let $m,n\in\N$ with $m=p_1^{k_1}\dots p_t^{k_t}$ the prime factorization of $m$. If $p_1,\dots,p_t$ are odd and non-Wieferich\footnote{See definition \ref{def:wieferich_prime}.}, then $$P(m,n)=\lcm\paren{p_1^{k_1-1}P(p_1,n),\dots ,p_t^{k_t-1}P(p_t,n)}.$$
\end{theorem*}

\paragraph{\textbf{Acknowledgements}}
I thank my friend Lucas Michel for telling me about the question of this paper.

\section{Basic results}\label{sec:basic_results}

The first result below allows us to compute iterations of $T$ in a very simple way. It will be used extensively throughout the paper.

\begin{proposition}\label{prop:dev}
Let $\textbf{a}\in\Z_m^n$, with $m,n\in\N$. For all $r\in\N$ and $i$ such that $0\leq i<n$,\[[T^r\textbf{a}]_i\equiv\sum_{j=0}^r \binom{r}{j} \textbf{a}_{i+j}\pmod{m}.\]
\end{proposition}
\begin{proof}
We prove this by induction on $r$. For $r=0$, the result is obvious. Suppose it holds for $r$. We show that it holds for $r+1$, by using Pascal's triangle formula and manipulating the sums as follows,
\begin{align*}
    [T^{r+1}\textbf{a}]_i&\equiv[T^r\textbf{a}]_i+[T^r\textbf{a}]_{i+1}\equiv\sum_{j=0}^r \binom{r}{j} \textbf{a}_{i+j}+\sum_{j=0}^r \binom{r}{j} \textbf{a}_{i+j+1}\\
    &\equiv\binom{r}{0}\textbf{a}_i+\sum_{j=1}^r \paren{\binom{r}{j}+\binom{r}{j-1}} \textbf{a}_{i+j}+\binom{r}{r}\textbf{a}_{i+r+1}\\
    &\equiv\sum_{j=0}^{r+1} \binom{r+1}{j} \textbf{a}_{i+j}\pmod{m},
\end{align*}
which completes the proof.
\end{proof}

\begin{definition}
We say that a T-sequence $(T^r\textbf{a})_{r\geq 0}$, where $\textbf{a}\in\Z^n_m$, \emph{vanishes} if there exists a positive integer $r$ such that $T^r\textbf{a}=0$.
\end{definition}
Recall that T-sequences are Ducci sequences if $m=2$. Thus it is well known that every T-sequence of $\Z^n_2$ vanishes if and only if $n$ is a power of $2$. It has first been proven by Ciamberlini and Marengoni in \cite{CM}, and it has been reproven many times since then \cite{Br,Eh}. Actually, this result still holds when $m$ is any power of $2$. This has been proven by Wong in \cite{Wo}. We give here a shorter proof using the notations we introduced and proposition \ref{prop:dev}. In \cite{AC1}, C. Avart shows a converse to this theorem for the base case $m=2$, stating that the only tuples that vanish are the tuples obtained by concatenation of several copies of a tuple of length a power of 2.

\begin{theorem}\label{thm:power_2}
If $m=2^c$ and $n=2^d$ for some positive integers $c$, $d$, then every T-sequence of $\Z^{n}_{m}$ vanishes, that is, $P(m,n)=1$. Reciprocally, if every T-sequence of $\Z_m^n$ vanishes, then $m$ is a power of $2$.
\end{theorem}
\begin{proof}
We first prove the case $c=1$. Let $\textbf{a}\in\Z^{2^d}_2$. Since $\binom{2^d}{j}$ is even for $0<j<2^d$ and by Proposition \ref{prop:dev}, we have
\[\cro{T^{2^d}\textbf{a}}_i\equiv\sum^{2^d}_{j=0}\binom{2^d}{j} \textbf{a}_{i+j}\equiv 2\textbf{a}_i\equiv 0\pmod{2}
\]
for all $i$ between $0$ and $n-1$.

We now proceed with a proof by induction. Assume the result holds for some integer $c$. We prove that it holds for $c+1$. Let $\textbf{a}\in\Z^{2^d}_{2^{c+1}}$. If considered over $\Z^{2^d}_{2^c}$, the T-sequence generated by $\textbf{a}$ vanishes. Let $r$ be an integer such that $T^r\textbf{a}\equiv\textbf{0}\pmod{2^c}$. Therefore $T^r\textbf{a}\equiv 2^c\textbf{u}\pmod{2^{c+1}}$ for some tuple $\textbf{u}$, which we can assume to consist only of $0$ and $1$'s. It follows from the base case and by linearity of $T$ that there exists some integer $r'$ such that $T^{r+r'}\textbf{a}\equiv 2^cT^{r'}\textbf{u}\equiv\textbf{0}\pmod{2^{c+1}}$, which concludes the proof.

The reciprocal follows from the proof of Proposition \ref{prop:div_order}.
\end{proof}

We denote by $H$ the \emph{left-shift} map \cite{Eh}, defined as \[H:\Z^n_m\rightarrow\Z^n_m:(a_0,a_1,\dots,a_{n-1})\mapsto(a_1,a_2,\dots,a_0).\] Thus $T=I+H$ where $I$ is the identity map. 
The following Lemma generalizes lemma $1$ in \cite{Eh}.

\begin{lemma}\label{lem:I+H}
If $p$ is a prime and $k$ is a positive integer, then $T^{p^k}=I+H^{p^k}$ as linear maps from $\Z^n_p$ into itself.
\end{lemma}
\begin{proof}
We have $T=I+H$ where $I$ is the identity map. The proof follows from Proposition \ref{prop:dev} and the fact that a prime $p$ divides $\binom{p^k}{j}$ for $0<j<p^k$.
\end{proof}

The next proposition is a generalization of Corollary 3 and Theorem 2 in \cite{Eh}.

\begin{proposition}\label{prop:prime_div}
Let $p$ be a prime and $n,K$ positive integers.
\begin{enumerate}
    \item If $p^K\equiv 1\pmod{n}$, then $P(p,n)$ divides $p^K-1$.
    \item If $p^K\equiv -1\pmod{n}$, then $P(p,n)$ divides $n(p^K-1)$.
\end{enumerate}
\end{proposition}
\begin{proof}
By Lemma \ref{lem:I+H}, we have:
\begin{enumerate}
    \item $T^{p^K}=I+H^{p^K}=I+H=T$.
    \item $T^{p^K}=I+H^{p^K}=I+H^{-1}=H^{-1}T$, hence $T^{np^K}=H^{-n}T^n=T^n$.
\end{enumerate}
\end{proof}

If $p$ and $n$ are coprime, then $K=O_p(n)$, the order of $p$ in $\Z_n$, always satisfies (1).

\section{Multiplicity of the period function}\label{sec:multiplicity}

In the following sections, we focus on the main question of this paper: can we deduce $P(m,n)$ from the prime factorization of $m=p_1^{k_1}\dots p_t^{k_t}$, knowing $P(p_i,n)$ for $i=1,\dots,t$?

The next proposition suggests a positive answer. We will often use it without reference. 

\begin{proposition}\label{prop:div}
If $d\mid m$, then $P(d,n)\mid P(m,n)$ for all $n\in\N$.
\end{proposition}
\begin{proof}
Let $r\in\N$ be large enough so that $\textbf{a}=T^r\textbf{e}\in\mathcal{C}_m^n\subset\mathcal{C}_d^n$. The congruence $T^{P(m,n)}\textbf{a}\equiv \textbf{a}\pmod{m}$ still holds modulo $d$ since $d\mid m$. The conclusion follows from Definition \ref{def:period}.
\end{proof}

The goal of the next few sections is to study this relation more precisely.

\begin{theorem}\label{thm:multiplicity}
If $m=\prod_{i=1}^t p_i$ with $p_1,\dots,p_t$ pairwise coprime, then \[P(m,n)=\lcm \paren{P(p_1,n),\dots,P(p_t,n)}.\]
\end{theorem}
\begin{proof}
We prove this for two coprime integers. The generalization for $t$ pairwise coprime integers follows by induction. By Proposition \ref{prop:basic}, we only need to consider the basic T-sequence.

Let $p,q$ be two coprime integers. We assume here that $r$ is large enough for $T^r\textbf{e}$ to be in the different cycles. By Definition \ref{def:period}, we have
\[\begin{cases}
    T^{r+L}\textbf{e}\equiv T^r\textbf{e}\pmod{p}\\
    T^{r+L}\textbf{e}\equiv T^r\textbf{e}\pmod{q}
\end{cases},\]
where $L=\lcm(P(p,n),P(q,n))$. Since $p$ and $q$ are coprime, it directly follows\footnote{If $a\equiv b\pmod{p}$ and $a\equiv b\pmod{q}$, then $a-b=cp=dq$ for some integers $c$, $d$. Hence $q$ divides $cp$, so $q$ divides $c$ by Euclid's lemma and we get $a-b=c'pq$ for some integer $c'$.} that $T^{r+L}\textbf{e}\equiv T^r\textbf{e}\pmod{pq}$. Hence $L$ is a multiple of the period $P(pq,n)$.

We now show that $L$ satisfies $(2)$ of Definition \ref{def:period}. Suppose $Q$ is such that $T^{r+Q}\textbf{e}\equiv T^r\textbf{e}\pmod{pq}$. In particular, $T^{r+Q}\textbf{e}\equiv T^r\textbf{e}\pmod{p}$; hence $Q$ is a multiple of $P(p,n)$. Similarly, $Q$ is a multiple of $P(q,n)$. Therefore, by definition of the least common multiple, $L\leq Q$, and $L=P(pq,n)$.
\end{proof}

With this theorem in our toolbox, we can restrict our attention to the periods modulo powers of primes. The question one may ask is whether we can deduce $P(p^k,n)$ from $P(p,n)$. We will shortly determine that it is (almost) the case.

\section{Order of 2 and Wieferich primes}\label{sec:order_2_Wieferich}

\begin{definition}
For a tuple $\textbf{a}\in\Z_m^n$, we write $\abs{\textbf{a}}$ for the \emph{sum of components} of $\textbf{a}$ modulo $m$.
\end{definition}

\begin{definition}\label{def:order}
For $m>2$ odd, we use $O(m)$ to denote the order of $2$ in $\Z_m$; that is, it is the smallest integer $k$ such that $2^k\equiv 1\pmod{m}$. Its existence follows from Euler's theorem.
\end{definition}

\begin{proposition}\label{prop:div_order}
If $m>2$ is odd, then $O(m)$ divides $P(m,n)$.
\end{proposition}
\begin{proof}
Let $r$ be a positive integer and $\textbf{a}$ be a tuple of $\Z_m^n$. By linearity, we have $\abs{T^r\textbf{a}}\equiv 2^r\abs{\textbf{a}}\pmod{m}$. If $L$ is the cycle length of the T-sequence generated by $\textbf{a}$ and $r$ is greater than the pre-period, then we must have $\abs{T^{r+L}\textbf{a}}\equiv 2^L\abs{T^r\textbf{a}}\pmod{m}$. Hence $2^L\equiv 1\pmod{m}$ or $\abs{T^r\textbf{a}}\equiv 0\pmod{m}$. Note that if $m$ is not a power of $2$, then $\abs{T^r\textbf{e}}\not\equiv 0\pmod{m}$ for all $r$. Considering the basic T-sequence thus implies that $2^{P(m,n)}$ must equal $1\pmod{m}$.
\end{proof}

\begin{definition}\label{def:wieferich_prime}
A prime $p$ is a \emph{Wieferich prime} if $2^{p-1}\equiv 1\pmod{p^2}$.
\end{definition}

Wieferich primes surprisingly occur in several number theoretical subjects \cite{HW}. It is believed that there are infinitely many such numbers. What is extraordinary about these is that we only know two of them, $1093$ and $3511$, and there are no other Wieferich primes below $10^{17}$ \cite{OEIS}. We will see in Section \ref{sec:prime_powers} that Wieferich primes are of considerable interest here. 

We can characterize Wieferich primes by the order of $2$ modulo $p^2$.
\begin{lemma}\label{lem:Wieferich}
A prime $p>2$ is a Wieferich prime if and only if $O(p)=O(p^2)$.
\end{lemma}
\begin{proof}
We first show that $O(p^2)$ equals either $O(p)$ or $pO(p)$. By definition, $2^{O(p^2)}\equiv 1\pmod{p^2}$. It also holds modulo $p$, so $O(p^2)$ is a multiple of $O(p)$. We have
\[2^{pO(p)}-1\equiv\paren{2^{O(p)}-1}\paren{2^{(p-1)O(p)}+2^{(p-2)O(p)}+\dots+2^{O(p)}+1}\equiv 0\pmod{p^2},\]
so $O(p^2)$ divides $pO(p)$. Thus it equals either $O(p)$ or $pO(p)$.

Suppose that $p$ is a Wieferich prime, i.e., $2^{p-1}\equiv 1\pmod{p^2}$. Then $O(p^2)$ divides $p-1$, since we cannot have $O(p^2)=pO(p)$, we have $O(p^2)=O(p)$. Conversely, if $O(p^2)=O(p)$, then $O(p^2)$ divides $p-1$, so $2^{p-1}\equiv 1\pmod{p^2}$.
\end{proof}

As the following proposition shows, the first part of the previous proof also holds for all primes $p>2$ and positive integers $k$, that is, $O(p^{k+1})$ is either $O(p^k)$ or $pO(p^k)$. In fact, as soon as it is the latter for one $k$, it is the latter for all subsequent $k$.

\begin{proposition}\label{prop:order_wief}
If $p>2$ is a prime and $k\in\N$, then we have:
\begin{enumerate}
    \item $O(p^{k+1})$ is either $O(p^k)$ or $pO(p^k)$.
    \item If $O(p^{k+1})=pO(p^k)$, then $O(p^{k+2})=pO(p^{k+1})$.
    \item If $p$ is a non-Wieferich prime, then $O(p^k)=p^{k-1}O(p)$.
\end{enumerate}
\end{proposition}
\begin{proof}
The proof of $(1)$ is similar to the first part of the proof of Lemma \ref{lem:Wieferich} and (3) follows from (1) and (2) by induction. We show (2):

Suppose $O(p^{k+1})=pO(p^k)$. Then $2^{O(p^k)}\equiv 1+lp^k\pmod{p^{k+1}}$ where $l\not\equiv 0\pmod{p}$. Hence $2^{O(p^k)}\equiv 1+lp^k+l'p^{k+1}\equiv 1+p^k(l+l'p)\pmod{p^{k+2}}$ for some $l'$. By the binomial theorem and the fact that $p$ divides $\binom{p}{j}$ for $0<j<p$, we have $2^{pO(p^k)}\equiv 1+p^{k+1}(l+l'p)\equiv 1+lp^{k+1}\pmod{p^{k+2}}$. Since $l\not\equiv 0\pmod{p}$, we have $2^{pO(p^k)}\not\equiv 1 \pmod{p^{k+2}}$, so $O(p^{k+2})$ must equal $p^2 O(p^k)$. This concludes the proof.
\end{proof}

For the known Wieferich primes, we have $O(p^3)=pO(p^2)$. Hence by (3) of the previous proposition, it follows that $O(p^k)=p^{k-2}O(p)$ for all $k\geq 2$.

\section{Period modulo powers of primes}\label{sec:prime_powers}

Propositions \ref{prop:div_order} and \ref{prop:order_wief} suggest a similar induction relation for the period function. Indeed, if $p$ is an odd prime and $n\in\N$, then $O(p^k)\mid P(p^k,n)$ for all positive integers $k$. The fact that $(O(p^k))_{k\in\N}$ eventually grows as a geometric sequence forces $(P(p^k,n))_{k\in\N}$ to behave in the same way.

\begin{theorem}\label{thm:induction}
If $p$ is a prime and $k,n\in\N$, then we have:
\begin{enumerate}
    \item $P(p^{k+1},n)$ is either $P(p^k,n)$ or $pP(p^k,n)$.
    \item If $k\geq 2$ and $P(p^{k+1},n)=pP(p^k,n)$, then $P(p^{k+2},n)=pP(p^{k+1},n)$.
    \item If $P(p^{N+1},n)=pP(p^N,n)$ for some $N\geq 2$, then $P(p^{N+k},n)=p^kP(p^N,n)$ for all $k\in\N$.
\end{enumerate}
\end{theorem}
\begin{proof}
In (1) and (2), we choose $r$ sufficiently large for $\textbf{a}=T^r\textbf{e}$ to be in a cycle of $\Z_{p^{k+1}}^n$ and $\Z_{p^{k+2}}^n$, respectively.

\medskip
\paragraph{(1)} Let $L=P(p^k,n)$. Since $\textbf{a}$ is in a cycle modulo $p^{k+1}$, it is also in a cycle modulo $p^k$ and $T^{P(p^{k+1},n)}\textbf{a}\equiv \textbf{a}\pmod{p^k}$. Hence $L$ divides $P(p^{k+1},n)$.

We have $T^L\textbf{a}\equiv \textbf{a}\pmod{p^k}$, so $T^L\textbf{a}=\textbf{a}+p^k\textbf{u}$ for some tuple $\textbf{u}$ that we can consider to be in $\Z_p^n$. By linearity of $T$, the tuple $p^k\textbf{u}=T^L\textbf{a}-\textbf{a}$ is in a cycle modulo $p^{k+1}$. It implies that $\textbf{u}$ is in a cycle in $\Z_p^n$, hence $T^L\textbf{u}=\textbf{u}+p\textbf{v}$ for some tuple $\textbf{v}$. Then, by linearity, $$T^{2L}\textbf{a}\equiv T^L\textbf{a}+p^kT^L\textbf{u}\equiv \textbf{a}+p^k\textbf{u}+p^k(\textbf{u}+p\textbf{v})\equiv \textbf{a}+2p^k\textbf{u}\pmod{p^{k+1}}.$$ By iterating $T^L$ on $\textbf{a}$, we then obtain $T^{pL}\textbf{a}\equiv\textbf{a}+pp^k\textbf{u}\equiv\textbf{a}\pmod{p^{k+1}}$. Hence $P(p^{k+1},n)$ divides $pL$.

Therefore, $P(p^{k+1},n)$ is either $L$ or $pL$.

\medskip
\paragraph{(2)} Suppose $k\geq 2$ and let $L=P(p^k,n)$, $L_1=P(p^{k+1},n)$ and $L_2=P(p^{k+2},n)$. Suppose $L_1=pL$. By (1), we know that $L_2$ is either $L_1$ or $pL_1$. We show that it equals the latter.

Since $\textbf{a}$ is in a cycle modulo $p^{k+2}$, it is also in a cycle modulo $p^{k+1}$ and $p^k$. Then $T^L\textbf{a}\equiv \textbf{a}\pmod{p^k}$, but this congruence does not hold modulo $p^{p+1}$, for we assumed that $L_1=pL$. This implies that $T^L\textbf{a}=\textbf{a}+p^k\textbf{u}$ where $\textbf{u}\not\equiv 0\pmod{p}$.

By linearity of $T$, the tuple $p^k\textbf{u}= T^L\textbf{a}-\textbf{a}$ is in a cycle modulo $p^{k+2}$, hence $\textbf{u}$ is in a cycle modulo $p^2$. The condition $k\geq 2$ implies that $P(p^2,n)$ divides $L=P(p^k,n)$. We then have, by the same argument as in (1),  $$T^{L_1}\textbf{a}\equiv T^{pL}\textbf{a}\equiv \textbf{a}+p^{k+1}\textbf{u}\not\equiv\textbf{a}\pmod{p^{k+2}}.$$
Therefore, $L_2\neq L_1$, so $L_2=pL_1$.

\medskip
\paragraph{(3)} This follows directly from (1) and (2) by induction.
\end{proof}

The following proposition exhibits the fact that the condition $k\geq 2$ is needed for (2) in the previous theorem to hold.

\begin{proposition}\label{prop:exception_3}
We have $P(2,3)=3$ and $P(2^k,3)=6$ for all positive integers $k>1$.
\end{proposition}
\begin{proof}
By Theorem \ref{prop:basic}, computing the first iterations of the basic T-sequence gives $P(2,3)=3$.

We prove that $P(2^k,3)=6$ for all positive integers $k>1$ by induction. The idea of the argument is to show that, for all $k>1$:
\begin{enumerate}
    \item The pre-period of the basic T-sequence of $\Z_{2^k}^3$ is $N_k=k+1$ and it has a cycle of the form 
    \begin{equation}\label{eq:cycle}
    \big((a,a,b),(d,c,c),(a,b,a),(c,c,d),(b,a,a),(c,d,c)\big) 
    \end{equation}
    where $a,b,c,d\in\Z_{2^k}$, with $a,b\leq 2^{k-1}\leq c,d$.
    \item The pre-period of the basic T-sequence of $\Z_{2^{k+1}}^3$ is $N_{k+1}=k+2$ and it has a cycle of the form $$\big((d,c,c),(a',b',a'),(c,c,d),(b',a',a'),(c,d,c),(a',a',b')\big)$$ where $a',b'\in\Z_{2^{k+1}}$ and $c,d\leq 2^k\leq a',b'$ (Note that $c$ and $d$ are those from (1)). In other words, it means that the cycle in $\Z_{2^{k+1}}^3$ starts at the second tuple of the cycle in $\Z_{2^k}^3$.
\end{enumerate}

For the base cases, $k=2$ and $3$, the cycles of the basic T-sequences are respectively $$\big((1,1,2),(2,3,3),(1,2,1),(3,3,2),(2,1,1),(3,2,3)\big)$$ and $$\big((2,3,3),(5,6,5),(3,3,2),(6,5,5),(3,2,3),(5,5,6)\big),$$ so (1) and (2) are satisfied.

Now let $k>1$ and suppose (1) is satisfied. Since $c,d\geq 2^{k-1}$, we have $(d,c,c)=2^{k-1}(1,1,1)+(d',c',c')$ for some $c',d'<2^{k-1}$. Then by linearity, $$2^k(1,1,1)+(d'+c',c'+c',d'+c')=T(d,c,c)\equiv (a,b,a)\pmod{2^k},$$ hence $(d'+c',c'+c',d'+c')=(a,b,a)$ since $a,b,c'+c',c'+d'$ are smaller than $2^k$. If we let $a'=2^k+a$ and $b'=2^k+b$, we have $T(d,c,c)=(a',b',a')$ and $T^2(d,c,c)\equiv (c,c,d)\pmod{2^{k+1}}$. By symmetry of this argument under cyclic permutations, we have $T^3(d,c,c)\equiv (b',a',a')$, $T^4(d,c,c)\equiv (c,d,c)$, $T^5(d,c,c)\equiv (a',a',b')$ and $T^6(d,c,c)\equiv(d,c,c)$, as desired. Then $N_{k+1}=N_k+1$ and (2) holds.
\end{proof}

The goal of the rest of this section is to show the main theorem of this paper, stated in the Introduction. To do that, we need to find \emph{base cases} in order to use (3) of Theorem \ref{thm:induction}. We first prove the case $p\nmid n$ in Proposition \ref{prop:period_p^2}. Then we will use combinatorial congruences to deduce the case $p\mid n$.

\begin{proposition}\label{prop:period_p^2}
Let $p>2$ be a non-Wieferich prime and $n\in\N$. If $n$ is not a multiple of $p$, then $P(p^2,n)=pP(p,n)$ and $P(p^3,n)=p^2P(p,n)$.
\end{proposition}
\begin{proof}
By Theorem \ref{thm:induction}, we know that $P(p^2,n)$ is either $P(p,n)$ or $pP(p,n)$. Given the fact that $p$ and $n$ are coprime, Proposition \ref{prop:prime_div} tells us that $P(p,n)$ divides $p^{O_p(n)}-1$, so it cannot be a multiple of $p$. However, by Proposition \ref{prop:div_order}, $P(p^2,n)$ is a multiple of $O(p^2)$, which equals $pO(p)$ for $p$ is non-Wieferich (Proposition \ref{prop:order_wief}). Therefore, $p$ divides $P(p^2,n)$ and we must have $P(p^2,n)=pP(p,n)$.

Similarly, $P(p^3,n)$ is either $pP(p,n)$ or $p^2P(p,n)$. Since $O(p^3)=p^2O(p)$ divides $P(p^3,n)$ and $pP(p,n)$ is not divisible by $p^2$, we must have $P(p^3,n)=p^2P(p,n)$.
\end{proof}

\begin{corollary}
If $p>2$ is a non-Wieferich prime and $n$ is not a multiple of $p$, then $P(p^k,n)=p^{k-1}P(p,n)$ for all $k\in\N$.
\end{corollary}
\begin{proof}
It follows from Proposition \ref{prop:period_p^2} and (3) in Theorem \ref{thm:induction}.
\end{proof}

To generalize Proposition \ref{prop:period_p^2} for any $n\in\N$, we first need to prove a few lemmas.
\medskip

The \emph{p-adic valuation} of an integer $n$ is the exponent of the largest power of $p$ that divides $n$. It is denoted by $v_p(n)$. We write $s_p(n)$ the sum of the digits of $n$ when written in base $p$. If $p$ is a prime, \emph{Legendre's formula} \cite{Le} states that $$v_p(n!)=\frac{n-s_p(n)}{p-1}.$$
Note that $v_p(ab)=v_p(a)+v_p(b)$ and $v_p(a/b)=v_p(a)-v_p(b)$ for all integers $a,b$.

The following Lemma generalizes the fact that $p$ divides $\binom{p^v}{n}$ for all $0<n<p^v$ but not for $n=0$ and $p^v$. We will use the cases $s=2$ and $3$ later in Proposition \ref{prop:base_case}.

\begin{lemma}\label{lem:Leg_formula}
Let $p$ be a prime and $v,s\in\N$ with $1\leq s\leq v$. We have $$\aco{n\in\aco{0,\dots,p^v}\colon p^{s}\nmid\binom{p^v}{n}}=\aco{mp^{v-s+1}\colon 0\leq m\leq p^{s-1}}.$$
\end{lemma}
\begin{proof}
It is clear that $0$ and $p^v$ belong to both sets. Let $0<n<p^v$ and $\sum_{i=0}^{v-1}n_ip^i$ its decomposition in base $p$. We show that $v_p\paren{\binom{p^v}{n}}=v-r$ where $r=\min\aco{i\colon n_i\neq 0}$. Since $0<n<p^v$, we have $0<r<v$.

By Legendre's formula, we have $v_p(p^v!)=\frac{p^v-1}{p-1}$ and $v_p(n!)=\frac{n-s_p(n)}{p-1}$ where $s_p(n)=\sum_i n_i$. We also have $$p^v-n =1+\sum_{i=0}^{v-1}(p-1-n_i)p^i=(p-n_r)p^r+\sum_{i=r+1}^{v-1}(p-1-n_i)p^i.$$
Thus, $$v_p((p^v-n)!)=\frac{p^v-n-\sum_{i=r+1}^{v-1}(p-1-n_i)-(p-n_r)}{p-1},$$
and
\begin{align*}
    v_p(n!(p^v-n)!) &= v_p(n!)+v_p((p^v-n)!) = \frac{1}{p-1}\paren{p^v-\sum_{i=r+1}^{v-1}(p-1)-p}\\
    &= \frac{p^v-(v-r-1)(p-1)-p}{p-1} = \frac{p^v-1-v(p-1)+r(p-1)}{p-1}\\&=\frac{p^v-1}{p-1}+r-v.
\end{align*}
Therefore, we obtain $v_p\paren{\binom{p^v}{n}}=v-r$, hence an integer $0<n<p^v$ belongs to the first set if and only if $v-r<s$, i.e., $r\geq v-s+1$, which happens if and only if $n$ is a multiple of $p^{v-s+1}$. This concludes the proof.
\end{proof}

The proof of the following lemma is due to Darij Grinberg and Victor Reiner ((12.69.3) in \cite{DV}).
\begin{lemma}\label{lem:darij}
Let $n\in\N$ and $p$ be a prime factor of $n$. For all $q\in\N$ and $r\in\Q$ such that $rn/p$ is an integer, we have $$\binom{qn}{rn}\equiv \binom{qn/p}{rn/p}\pmod{p^{v_p(n)}}.$$
\end{lemma}
\begin{proof}
The argument consists in counting the $(rn)$-elements subsets of the set $\Z_{qn}$. It is clearly $\binom{qn}{rn}$.

At the same time, the subsets fall into two classes:
\begin{enumerate}
    \item The subsets which are invariant under the permutation $i\mapsto i+qn/p$ of $\Z_{qn}$.
    \item The other ones.
\end{enumerate}
Say there are $N_1$ and $N_2$ subsets in the first and second class, respectively.

If a $(rn)$-elements subset $S$ belongs to the first class, then the intersection $S\cap\{0,1,\dots,qn/p-1\}$ must have $rn/p$ elements, which uniquely determine all of $S$ by iterating the permutation given above. Thus the first class contains $N_1=\binom{qn/p}{rn/p}$ elements.

Besides, the permutation $\phi:i\mapsto i+qn/p^{v_p(n)}$ of $\Z_{qn}$ acts on the subsets of the second class, splitting them into orbits. Its $p^{v_p(n)}$-power acts trivially on the subsets of the second class. Then, the size of each orbit divides $p^{v_p(n)}$. Suppose that the size $\abs{\mathcal{O}}$ of an orbit $\mathcal{O}$ is a proper divisor of $p^{v_p(n)}$. Then $\abs{\mathcal{O}}$ divides $p^{v_p(n)-1}$, so $\phi^{p^{v_p(n)-1}}$ acts trivially on $\mathcal{O}$, hence elements of this orbit are subsets of the first class, a contradiction. Then every orbit has size $p^{v_p(n)}$.

Since the set of all second class subsets is the union of these orbit, it has size $N_2$ divisible by $p^{v_p(n)}$.

Therefore, we have $\binom{qn}{rn}=N_1+N_2\equiv\binom{qn/p}{rn/p}\pmod{p^{v_p(n)}}$ and the proof is complete.
\end{proof}

To prove the following lemma, we shall introduce \emph{Babbage's theorem} (Theorem 1.12 in \cite{DG}). It states that for any prime $p$ and integers $a,b\geq 0$, we have $\binom{ap}{bp}\equiv\binom{a}{b}\pmod{p^2}$. Note that if $p\geq 5$, we can replace $\pmod{p^2}$ by $\pmod{p^3}$, thus strengthening the result. This case is known as \emph{Wolstenholme's theorem}.

\begin{lemma}\label{lem:babbage}
If $p$ is a prime, $v\geq 1$ and $0\leq j\leq p$, then $$\binom{p^v}{jp^{v-1}}\equiv\binom{p}{j}\pmod{p^2}$$
\end{lemma}
\begin{proof}
It follows directly from Babbage's theorem by induction.
\end{proof}

\begin{lemma}\label{lem:wolstenholme}
If $p$ is a prime, $v\geq 2$ and $0\leq j\leq p^2$, then $$\binom{p^v}{jp^{v-2}}\equiv\binom{p^2}{j}\pmod{p^3}$$
\end{lemma}
\begin{proof}
For $p\geq 5$, the proof follows from Wolstenholme's theorem by induction.

Now we consider the cases $p=2$ and $p=3$. If $v=2$, it is direct. If $v>2$, then Lemma \ref{lem:darij} with $q=1$, $n=p^v$ (then $v_p(n)=v$) and $r=j/p^2$ gives $$\binom{p^v}{jp^{v-2}}=\binom{qn}{rn}\equiv \binom{qn/p}{rn/p}\equiv \binom{p^{v-1}}{jp^{v-3}}\pmod{p^v}$$ and the congruence also holds modulo $p^3$ since $v\geq 3$. Then the proof proceeds by induction.
\end{proof}

Now we can apply these lemmas in order to prove the following proposition, which gives us base cases to apply Theorem \ref{thm:induction}.

\begin{proposition}\label{prop:base_case}
Let $p$ be a prime and $n\in\N$ with $n=p^vn'$, where $v=v_p(n)$. Then we have
\begin{enumerate}
    \item $P(p,n)=p^vP(p,n')$. If $p=2$, it holds only if $n'\neq 1$.
    \item If $p>2$ is a non-Wieferich prime, then $P(p^2,n)=p^vP(p^2,n')$.
    \item If $p>2$ is a non-Wieferich prime, then $P(p^3,n)=p^vP(p^3,n')$.
\end{enumerate}
\end{proposition}
\begin{proof}
The idea of this proof is to study the behavior of a T-sequence $(T^i\textbf{a})_{i\in\N}$ of $\Z_m^n$ (where $m=p,p^2,p^3$, respectively) by studying the behavior of the T-sequence generated by a \emph{subtuple} of $\textbf{a}$, which is a T-sequence of smaller tuples that we better understand.

For that purpose, we introduce a family of functions $S_r$, $0\leq r\leq v$, that extract an interesting subtuple from a given tuple. For $0\leq r\leq v$, let $$S_r:\Z_m^n\rightarrow\Z_m^{p^rn'}:(a_0,\dots,a_{n-1})\mapsto (a_0,a_{p^{v-r}},a_{2p^{v-r}},\dots,a_{(p^rn'-1)p^{v-r}}).$$
\medskip

\paragraph{(1)} Here we use $S_0$. For $\textbf{a}\in\Z_p^n$, the subtuple $S_0(\textbf{a})$ is in $\Z_p^{n'}$.

By Lemma \ref{lem:I+H}, we have $T^{p^v}\equiv I+H^{p^v}\pmod{p}$, hence $S_0(T^{p^v}\textbf{a})\equiv TS_0(\textbf{a})\pmod{p}$ for any tuple $\textbf{a}\in\Z_p^n$. Considering the basic tuple $\textbf{e}$ of $\Z_p^n$, it gives $S_0(\textbf{e})=\textbf{e'}$ where $\textbf{e'}$ is the basic tuple of $\Z_p^{n'}$. Note that components of $\textbf{e}$ that are not components of $\textbf{e'}$ remain zero after any number of iterations of $T^{p^v}$, hence the behavior of the T-sequence $(T^{rp^v}\textbf{e})_{r\geq 0}$ is entirely determined by the behavior of $(T^r\textbf{e'})_{r\geq 0}$. Since the cycle length of the latter is $P(p,n')$, the cycle length of $(T^r\textbf{e})_{r\geq 0}$ is $p^vP(p,n')$.

If $p=2$, note that this argument only holds if the T-sequences do not vanish. This explains the additional condition $n'\neq 1$ in this case.

\medskip
\paragraph{(2)} Here we suppose $p>2$ is a non-Wieferich prime. We first show that $P(p^2,pn')=pP(p^2,n')$ and then that $P(p^2,n)=p^{v-1}P(p^2,pn')$. 

To prove the first part, we use $S_0:\Z_{p^2}^{pn'}\rightarrow\Z_{p^2}^{n'}$. By Lemmas \ref{lem:Leg_formula} and \ref{lem:babbage}, we have
\begin{align*}
    [T^{p^2}\textbf{a}]_i\equiv\sum_{j=0}^{p^2} \binom{p^2}{j} \textbf{a}_{i+j}\equiv\sum_{j=0}^{p} \binom{p^2}{jp} \textbf{a}_{i+jp} \equiv \sum_{j=0}^{p} \binom{p}{j} \textbf{a}_{i+jp}\pmod{p^2},
\end{align*}
which implies that $S_0(T^{p^2}\textbf{a})\equiv T^pS_0(\textbf{a})\pmod{p^2}$. Since $p$ is an odd non-Wieferich prime, $p$ divides $P(p^2,n')$ by Proposition $\ref{prop:period_p^2}$. Thus we obtain $P(p^2,pn')=p^2p^{-1}P(p^2,n')=pP(p^2,n')$.

We now use $S_1$. For $\textbf{a}\in\Z_p^n$, the subtuple $S_1(\textbf{a})$ is in $\Z_{p^2}^{pn'}$. We also have
\begin{align*}
    [T^{p^v}\textbf{a}]_i&\equiv\sum_{j=0}^{p^v} \binom{p^v}{j} \textbf{a}_{i+j}\equiv\sum_{j=0}^{p} \binom{p^v}{jp^{v-1}} \textbf{a}_{i+jp^{v-1}} \equiv \sum_{j=0}^{p} \binom{p}{j} \textbf{a}_{i+jp^{v-1}}\pmod{p^2},
\end{align*}
where the second and third equalities follow from Lemmas \ref{lem:Leg_formula} and \ref{lem:babbage}, respectively. Thus, $S_1(T^{p^v}\textbf{a})\equiv T^pS_1(\textbf{a})\pmod{p^2}$. Therefore we obtain $P(p^2,n)=p^vp^{-1}P(p^2,pn')=p^{v-1}P(p^2,pn')=p^vP(p^2,n')$ as desired, where the last equality follows from the first part. 

\medskip
\paragraph{(3)} Suppose $p>2$ is a non-Wieferich prime. To begin with, we suppose $v>1$. We first show that $P(p^3,p^2n')=p^2P(p^3,n')$ and then that $P(p^3,n)=p^{v-2}P(p^3,p^2n')$.

First, we use $S_0$. For $\textbf{a}\in\Z_{p^3}^{p^2n'}$, the subtuple $S_0(\textbf{a})$ is in $\Z_{p^3}^{n'}$. By Lemmas \ref{lem:Leg_formula} and \ref{lem:wolstenholme}, we have
\begin{align*}
    [T^{p^4}\textbf{a}]_i&\equiv\sum_{j=0}^{p^4} \binom{p^4}{j} \textbf{a}_{i+j}\equiv\sum_{j=0}^{p^2} \binom{p^4}{jp^2} \textbf{a}_{i+jp^2} \equiv \sum_{j=0}^{p^2} \binom{p^2}{j} \textbf{a}_{i+jp^2}\pmod{p^3},
\end{align*}
hence $S_0(T^{p^4}\textbf{a})\equiv T^{p^2}S_0(\textbf{a})\pmod{p^3}$. By Proposition \ref{prop:period_p^2} (it is why we consider $p>2$ non-Wieferich), $p^2$ divides $P(p^3,n')$, thus we obtain $P(p^3,p^2n')=p^4p^{-2}P(p^3,n')=p^2P(p^3,n')$.

Now we use $S_2$. For $\textbf{a}\in\Z_p^n$, the subtuple $S_2(\textbf{a})$ is in $\Z_p^{p^2n'}$.
First, using Lemmas \ref{lem:Leg_formula} and \ref{lem:wolstenholme}, we obtain
\begin{align*}
    [T^{p^v}\textbf{a}]_i&\equiv\sum_{j=0}^{p^v} \binom{p^v}{j} \textbf{a}_{i+j}\equiv\sum_{j=0}^{p^2} \binom{p^v}{jp^{v-2}} \textbf{a}_{i+jp^{v-2}} \equiv \sum_{j=0}^{p^2} \binom{p^2}{j} \textbf{a}_{i+jp^{v-2}}\pmod{p^3},
\end{align*}
hence $S_2(T^{p^v}\textbf{a})\equiv T^{p^2}S_2(\textbf{a})\pmod{p^3}$. Together with the first part, since $p^2$ divides $P(p^3,p^2n')$ (by Proposition \ref{prop:period_p^2}), we get $P(p^3,n)=p^{v}p^{-2}P(p^3,p^2n')=p^vP(p^3,n')$ as desired.

To conclude the proof, we consider the case $v=1$. We have to show that $P(p^3,pn')=pP(p^3,n')$. We use $S_0$. For $\textbf{a}\in\Z_{p^3}^{pn'}$, the subtuple $S_0(\textbf{a})$ is in $\Z_{p^3}^{n'}$. As above, using Lemmas \ref{lem:Leg_formula} and \ref{lem:wolstenholme}, we have $[T^{p^3}\textbf{a}]_i\equiv \sum_{j=0}^{p^2} \binom{p^2}{j} \textbf{a}_{i+jp}\pmod{p^3}$, hence $S_0(T^{p^3}\textbf{a})\equiv T^{p^2}S_0(\textbf{a})\pmod{p^3}$. By Proposition \ref{prop:period_p^2}, $p^2$ divides $P(p^3,n')$. Thus we obtain $P(p^3,pn')=p^3p^{-2}P(p^3,n')=pP(p^3,n')$, which concludes.
\end{proof}

\begin{theorem}\label{thm:power_primes}
If $p>2$ is a non-Wieferich prime, we have $P(p^k,n)=p^{k-1}P(p,n)$ for all positive integers $k$ and $n$.
\end{theorem}
\begin{proof}
Let $k,n\in\N$. Write $n=p^vn'$ where $v=v_p(n)$. We show that (1) $P(p^2,n)=pP(p,n)$ and (2) $P(p^3,n)=p^2P(p,n)$.

\medskip
\paragraph{(1)} Points (1) and (2) of Proposition \ref{prop:base_case} yield $P(p,n)=p^vP(p,n')$ and $P(p^2,n)=p^vP(p^2,n')$, respectively. Since $p>2$ is non-Wieferich and coprime with $n'$, we have $P(p^2,n')=pP(p,n')$ by Proposition \ref{prop:period_p^2}. The conclusion follows directly.

\medskip
\paragraph{(2)} The argument is the same, using (3) of Proposition \ref{prop:base_case} instead of (2).

Therefore, we complete the proof by Theorem \ref{thm:induction}.
\end{proof}

We are now able to prove the main result of this paper.

\begin{theorem}
Let $m,n\in\N$ with $m=p_1^{k_1}\dots p_t^{k_t}$ the prime factorization of $m$. If $p_1,\dots,p_t$ are odd and non-Wieferich, then $$P(m,n)=\lcm\paren{p_1^{k_1-1}P(p_1,n),\dots ,p_t^{k_t-1}P(p_t,n)}.$$
\end{theorem}
\begin{proof}
It follows from Theorems \ref{thm:multiplicity} and \ref{thm:power_primes}.
\end{proof}

At this point, a question one may ask is whether we can generalize this result for $p=2$ and Wieferich primes. The proofs above are considerably dependant on Proposition \ref{prop:period_p^2}, which itself is dependant on Proposition \ref{prop:order_wief}. Therefore, it would not be possible to use the same method to obtain similar results for these special primes.

However, F. Breuer shows in Theorem 8.2 of \cite{Br2} that a variant of Theorem \ref{thm:power_primes} holds for $2$ and Wieferich primes, namely that $P(p^k,n)=p^{\max(0,k-1-t)}P(p,n)$ for some integer $t$. That is, these special cases eventually behave as we would expect. Whether it is possible or not to derive such results with an elementary method, similar to the ones used here, remains an open question.

\section{Characterization of tuples in a cycle}\label{sec:characterization_cycles}

In this section we try to characterize tuples of $\Z_m^n$ that belong to $\mathcal{C}_m^n$. By theorem \ref{thm:power_2}, we already know that $\mathcal{C}_m^n=\{\textbf{0}\}$ if $m$ and $n$ are powers of $2$.

The linear map $T$ is represented in the standard basis by the matrix \[\begin{pmatrix}
1&1&0&\dots&0\\
0&1&1&\dots&0\\
&\vdots&&\ddots&\vdots\\
1&0&0&\dots&1
\end{pmatrix}.\]
Note that $\det(T)=0$ when $n$ is even and $\det(T)=2$ when $n$ is odd. This simple fact yields the following proposition.
\begin{proposition}\label{prop:cycles_nodd}
Let $m>2$ and $n>0$ be two odd integers. Then $\mathcal{C}_m^n=\Z_m^n$.
\end{proposition}
\begin{proof}
Since $\det(T)=2$ is invertible\footnote{An integer $x$ is invertible modulo $m$ if and only if $x$ and $m$ are coprime.} modulo $m$, the matrix $T$ is invertible\footnote{A matrix is invertible if and only if its determinant is invertible.}, hence $T$ is bijective. Consequently, we can unambiguously move backward in the eventually periodic T-sequence determined by a tuple $\textbf{a}$ of $\Z^n_m$, so $\textbf{a}$ belongs to a cycle.
\end{proof}

For $n$ even, things are a bit more complicated and require the introduction of a few new notations. We denote the alternating sum of components of a tuple $\textbf{a}\in\Z^n_m$ by $$\sigma(\textbf{a})=\sum_{i=0}^{n-1}(-1)^ia_i\mod m.$$

\begin{proposition}\label{prop:cycles_neven}
Let $m>2$ be odd and $n>0$ be even. If $m$ and $n$ are coprime, then a tuple $\textbf{a}\in\Z^n_m$ belongs to a cycle if and only if $\sigma(\textbf{a})=0$.
\end{proposition}
\begin{proof}
We first show that the condition is sufficient. Let $\textbf{a}\in\Z^n_m$ be such that $\sigma(\textbf{a})=0$. Finding a preimage $\textbf{x}$ of $\textbf{a}$ is equivalent to solving the system
\begin{equation}\label{eq:syst}
\begin{cases}
    x_0+x_1\equiv a_0,\\
    x_1+x_2\equiv a_1,\\
    \quad\vdots\\
    x_{n-1}+x_0\equiv a_{n-1},
\end{cases}\textrm{which is equivalent to } 
\begin{cases}
    x_1\equiv a_0-x_0,\\
    x_2\equiv a_1-a_0+x_0,\\
    x_3\equiv a_2-a_1+a_0-x_0,\\
    \quad\vdots\\
    x_{n-1}\equiv a_{n-2}-a_{n-3}+\dots+a_0-x_0,\\
    x_0\equiv a_{n-1}-a_{n-2}+\dots+a_1-a_0+x_0.
\end{cases}
\end{equation}
All components of $\textbf{x}$ are determined by the value chosen for $x_0$ and the last equation is satisfied since $\sigma(\textbf{a})=0$. Thus, $\textbf{a}$ has exactly $m$ preimages. Moreover, $m$ and $n$ are coprime, so $n$ is invertible, hence the equation $\sigma(\textbf{x})\equiv nx_0-(n-1)a_0+(n-2)a_1-\dots+2a_{n-3}-a_{n-2}\equiv 0\pmod{m}$ has exactly one solution $x_0$. Thus, $\textbf{a}$ has exactly one preimage $\textbf{x}$ with $\sigma(\textbf{x})=0$. 

Therefore, the map $T$ restricted to the set $\aco{\textbf{a}\in\Z^n_m:\sigma(\textbf{a})=0}$ is a one-to-one correspondence and we can conclude as in Proposition \ref{prop:cycles_nodd}.

The last equation of (\ref{eq:syst}) shows that the condition is necessary.
\end{proof}

The following result, concerning $\Z_2^n$, is shown in \cite{Lu} and in \cite{MST} for the Ducci map, which is the same as our map in this case.

\begin{proposition}\label{prop:cycles_2}
In $\Z^n_2$, we have $\im(T)=\aco{\textbf{a}\in\Z^n_2 : \abs{\textbf{a}}=0}$ and every $\textbf{a}\in\im(T)$ has exactly two preimages. For odd $n$, a tuple $\textbf{a}\in\Z_2^n$ belongs to a cycle if and only if $\abs{\textbf{a}}=0$.
\end{proposition}

If $n$ is even, the two preimages of a tuple $\textbf{a}\in\im(T)$ are either both in $\im(T)$ or both in $\Z^n_2\backslash\im(T)$. Thus we have to find a way to characterize tuples of $\im(T)$ that have preimages in $\im(T)$.

That has actually already been done by Ludington-Young in \cite{Lu,LY}, where the following definition and theorem \ref{thm:Lu} come from. We only consider here tuples of $\Z^n_2$. A tuple $\textbf{a}$ is \emph{even} if $\abs{\textbf{a}}=0$. Suppose $n=2^rn'$ where $n'$ is odd, we say a tuple $\textbf{a}\in\Z^n_2$ is \emph{r-even} if
\[\sum^{k-1}_{i=0}a_{2^ri+j}\equiv 0\pmod{2}\]for $j=0,\dots,2^r-1$. For example, if $n=12$, then $\textbf{a}$ is $2$-even if 
\[a_0+a_4+a_8\equiv 0,\quad a_1+a_5+a_9\equiv 0,\quad a_2+a_6+a_{10}\equiv 0\quad\textrm{and}\quad a_3+a_7+a_{11}\equiv 0.\]
\begin{theorem}\label{thm:Lu}
Let $n=2^rn'$ with $n'$ odd. A tuple of $\Z^n_2$ belongs to a cycle if and only if it is $r$-even.
\end{theorem}
Proposition \ref{prop:cycles_2} turns out to be the special case $r=0$ of this theorem. Indeed, a tuple $\textbf{a}$ is $0$-even if and only if $\abs{\textbf{a}}=0$.

We now generalize this characterization to odd primes. 
\begin{definition}
Let $n=p^rn'$ with $r=v_p(n)$. We say a tuple $\textbf{a}$ of $\Z^n_p$ is \emph{even} if $\sigma(\textbf{a})=0$. We note \[\sigma_j(\textbf{a})=\sum^{n'-1}_{i=0}(-1)^ia_{p^ri+j}\mod p\] for $j=0,\dots,p^r-1$. We say $\textbf{a}$ is \emph{r-even} if $\sigma_j(\textbf{a})=0$ for all $j$.
\end{definition}

\begin{theorem}\label{thm:characterization}
Let $p>2$ be a prime and $n=p^rn'$ even with $r=v_p(n)$. A tuple of $\Z^n_p$ belongs to a cycle if and only if it is r-even.
\end{theorem}
\begin{proof}
The congruences below are all modulo $p$.

First, note that $T\textbf{a}$ is r-even if $\textbf{a}$ is r-even. Indeed, if $\textbf{a}$ is r-even, then $\sigma_j(T\textbf{a})\equiv \sigma_j(\textbf{a})+\sigma_{j+1}(\textbf{a})\equiv 0$ for each $j$. By Lemma \ref{lem:I+H}, we have that $T^{p^r}\textbf{e}\equiv \textbf{e}+H^{p^r}\textbf{e}$ is r-even. Thus, every tuple of a cycle must be r-even by linearity of r-evenness, hence the condition is necessary.

We now show that it is sufficient. Let $\textbf{a}\in\Z^n_p$ be r-even. By (\ref{eq:syst}) (it is here that we use the condition that $n$ is even) and since r-evenness implies evenness, the tuple $\textbf{a}$ has $p$ preimages. Let $\textbf{b}^{(0)}$ be one of these, hence all preimages are given by $\textbf{b}^{(l)}=\textbf{b}^{(0)}+(l,-l,\dots,l,-l)$ for $l=0,\dots,p-1$. To simplify notation, we write $\sigma_j$ for $\sigma_j\paren{\textbf{b}^{(0)}}$.

Since $a_i\equiv b^{(0)}_i+b^{(0)}_{i+1}$, \[0\equiv\sigma_0(\textbf{a})\equiv\sum^{n'-1}_{i=0}(-1)^ia_{p^ri}\equiv\sum^{n'-1}_{i=0}(-1)^i\paren{b^{(0)}_{p^ri}+b^{(0)}_{p^ri+1}}\equiv\sigma_0+\sigma_1\]
and, similarly, $\sigma_1+\sigma_2$, $\sigma_2+\sigma_3$,$\dots$,$\sigma_{p^r-2}+\sigma_{p^r-1}$ are all $\equiv 0$ and $\sigma_{p^r-1}-\sigma_{0}$ too\footnote{The negative sign comes from the equality $(-1)^ib^{(0)}_{p^ri+p^r-1+1}=(-1)^ib^{(0)}_{p^r(i+1)}$.}. Hence, \[\sigma_0\equiv -\sigma_1\equiv\sigma_2\equiv -\sigma_3\equiv\dots\equiv -\sigma_{p^r-2}\equiv\sigma_{p^r-1}.\]

Since $n'$ is invertible modulo $p$, the equation $\sigma_0\paren{\textbf{b}^{(l)}}\equiv\sigma_0+n'l\equiv 0\pmod{p}$ has exactly one solution $l\equiv -\sigma_0/n'\pmod{p}$, hence $\textbf{b}^{(l)}$ is the only r-even preimage of $\textbf{a}$.

Therefore, the map $T$ restricted to the set of r-even tuples of $\Z^n_p$ is bijective and the proof is complete.
\end{proof}

\section{Open questions}

We can see the iterations of the map $T$ on the set $\mathcal{C}_m^n$ (which is bijective when restricted to $\mathcal{C}_m^n$) as the action of the group $\aco{T^k\colon k\in\Z}$ on this set, splitting it into orbits. What are the possible sizes for these orbits? The tuple $(0,\dots,0)$ has an orbit of size 1, whereas the basic tuple, after enough iterations, generates an orbit of size $P(m,n)$. What are the values between $1$ and $P(m,n)$ that are the size of some orbit?

\medskip
Further questions arise naturally. What is the largest pre-period that can happen? How can the results of this paper be generalized to any linear map of $\Z_m^n$? Do there exist explicit formulas to find $P(p,n)$ for every prime $p$ and positive integer $n$?

\medskip

MSC2010: 05A10, 11B50, 11B75

\medskip

\end{document}